\documentclass[12pt, reqno]{amsart}
\usepackage{amsmath, amsthm, amscd, amsfonts, amssymb, graphicx, color}
\usepackage[bookmarksnumbered, colorlinks, plainpages]{hyperref}
\hypersetup{colorlinks=true,linkcolor=red, anchorcolor=green, citecolor=cyan, urlcolor=red, filecolor=magenta, pdftoolbar=true}

\usepackage{tensor}

\newtheorem{theorem}{Theorem}[section]

\newtheorem{proposition}[theorem]{Proposition}
\newtheorem{corollary}[theorem]{Corollary}
\theoremstyle{definition}
\newtheorem{definition}[theorem]{Definition}

\theoremstyle{remark}
\newtheorem{remark}[theorem]{Remark}
\numberwithin{equation}{section}

\usepackage{fullpage}

\begin{document}

\setcounter{page}{1}

\title[FPT on $G$-metric spaces]{ ($2k+1)^{th}$-order Fixed points in $G$-metric spaces}

\author[Ya\'e Ulrich Gaba]{Ya\'e Ulrich Gaba$^{1,2}$}
\author{COLLINS AMBURO AGYINGI$^{2}$}

\address{$^{1}$Institut de Math\'ematiques et de Sciences Physiques (IMSP)/UAC,
	Porto-Novo, B\'enin.}

\email{\textcolor[rgb]{0.00,0.00,0.84}{gabayae2@gmail.com
}}

\address{$^{2}$ Department of Mathematical Sciences, North West University, Private Bag
	X2046, Mmabatho 2735, South Africa.} 
\email{\textcolor[rgb]{0.00,0.00,0.84}{: collins.agyingi@nwu.ac.za
}}



\begin{abstract}
	
	We establish 
	three major fixed-point theorems for functions satisfying an odd  power type contractive condition in G-metric spaces.
	 We first consider the case of a single mapping, followed by that of a triplet of mappings and we conclude by the case of a family of mappings. The results we obtain in this
	 article extend similar ones already present in the literature. 
	
\end{abstract} 

\maketitle

\section{Introduction and Preliminaries}

$G$-metric fixed point theory is a growing area in the mathematical analysis because of its applications in areas like optimization theory and differential equations. It is worth pointing out that $G$-metrics appeared as a correction (provided by Mustafa\cite{Mustafa}) of the $D$-metric space theory introduced by Dhage\cite{d}. Indeed it was shown (see \cite{sim}) that certain theorems involving Dhage's $D$-metric spaces are flawed, and most of the results claimed by Dhage and others are invalid. In the literature, there are many results dealing with $G$-metric spaces and those can be read in \cite{Gaba7,r4,sha,v}. In the present manuscript, as in most papers dealing with $G$-metric fixed points, the key tool lies on the study
of the character of the sequence of iterates $\{T^nx\}_{n=0}^\infty$ (resp. $\{T_i(x_{i-1)}\}_{i=0}^\infty$) where $T: X \to X$ (resp.  $T_i: X \to X$ ), $x\in X$ and $X$ a complete $G$-metric
 space, is (resp. are) the map (resp. maps) under consideration.

Elementary facts about $G$-metric spaces as well as their properties related to can be found in \cite{Gaba7,Mustafa} and the references therein. We give here a summary of these requirements.
\begin{definition}\label{def1} (See \cite[Definition 3]{Mustafa})
Let $X$ be a nonempty set, and let the function $G:X\times X\times X \to [0,\infty)$ satisfy the following properties:
\begin{itemize}
\item[(G1)] $G(x,y,z)=0$ if $x=y=z$ whenever $x,y,z\in X$;
\item[(G2)] $G(x,x,y)>0$ whenever $x,y\in X$ with $x\neq y$;
\item[(G3)] $G(x,x,y)\leq G(x,y,z) $ whenever $x,y,z\in X$ with $z\neq y$;
\item[(G4)] $G(x,y,z)= G(x,z,y)=G(y,z,x)=\ldots$, (symmetry in all three variables);

\item[(G5)]
$$G(x,y,z) \leq [G(x,a,a)+G(a,y,z)]$$ for any points $x,y,z,a\in X$.
\end{itemize}
Then $(X,G)$ is called a \textbf{$G$-metric space}.

\end{definition}

\begin{definition}(See \cite{Mustafa})
	Let $(X, G)$ be a $G$-metric space, and let $\{x_n \}$ be a
	sequence of points of $X$, therefore, we say that $(x_n )$ is $G$-convergent to
	$x \in X$ if $\lim_{n,m\to \infty} G (x, x_n , x_m ) = 0,$ that is, for any $\varepsilon > 0$, there exists $N \in \mathbb{N}$ such that $G (x, x_n , x_ m ) < \varepsilon$, for all, $n, m \geq N$. We call $x$ the limit of the sequence and write $x_n \to x$ or $\lim_{n\to \infty} x_n = x$.
	
\end{definition}

\vspace*{0.2cm}

\begin{proposition}\label{prop1} (Compare \cite[Proposition 6]{Mustafa})
Let $(X,G)$ be a $G$-metric space. Define on $X$ the metric  $d_G$ by $d_G(x,y)= G(x,y,y)+G(x,x,y)$ whenever $x,y \in X$. Then for a sequence $(x_n) \subseteq X$, the following are equivalent
\begin{itemize}
\item[(i)] $(x_n)$ is $G$-convergent to $x\in X.$

\item[(ii)] $\lim_{n,m \to \infty}G(x,x_n,x_m)=0.$

\item[(iii)]  $\lim_{n \to \infty}d_G(x_n,x)=0$.

\item[(iv)]$\lim_{n \to \infty}G(x,x_n,x_n)=0.$ 

\item[(v)]$\lim_{n \to \infty}G(x_n,x,x)=0.$ 
\end{itemize}

\end{proposition}

\vspace*{0.2cm}

\begin{definition}(See \cite{Mustafa})
	Let $(X, G)$ be a $G$-metric space. A sequence $\{x_n \}$ is
	called a $G$-Cauchy sequence if for any $\varepsilon > 0$, there is $N \in \mathbb{N}$ such that
	$G (x_n , x_m , x_l ) < \varepsilon$ for all $n, m, l \geq N$, that is $G (x_n , x_m , x_l ) \to 0$ as $n, m,l \to +\infty$.
\end{definition}

\vspace*{0.2cm}

\begin{proposition}(Compare \cite[Proposition 9]{Mustafa})

In a $G$-metric space $(X,G)$, the following are equivalent
\begin{itemize}
\item[(i)] The sequence $(x_n) \subseteq X$ is $G$-Cauchy.

\item[(ii)] For each $\varepsilon >0$ there exists $N \in \mathbb{N}$ such that $G(x_n,x_m,x_m)< \varepsilon$ for all $m,n\geq N$.

\end{itemize}

\end{proposition}

\vspace*{0.2cm}

\begin{definition} (Compare \cite[Definition 4]{Mustafa})
 A $G$-metric space $(X,G)$ is said to be symmetric if 
 $$G(x,y,y) = G(x,x,y), \ \text{for all } x,y \in X.$$

\end{definition}

\vspace*{0.2cm}

\begin{definition} (Compare \cite[Definition 9]{Mustafa})
 A $G$-metric space $(X,G)$ is $G$-complete if every $G$-Cauchy sequence of elements of $(X,G)$ is $G$-convergent in  $(X,G)$. 

\end{definition}

\begin{theorem}\label{theorem1}(See \cite{Mustafa})
	A $G$-metric $G$ on a $G$-metric space $(X, G)$ is continuous on its three
	variables.
\end{theorem} 

\vspace*{0.2cm}

We conclude this introductory part with:

\vspace*{0.2cm}

\begin{definition}\label{vetro}(Compare \cite[Definition 2.1]{v})
	For a sequence $(a_n)_{n\geq 1}$ of nonnegative real numbers, the series $\sum_{n=1}^{\infty}a_n$ is an $\alpha$-series if there exist $0<\lambda<1$ and $n(\lambda) \in \mathbb{N}$ such that $$\sum_{i=1}^{L} a_i \leq \lambda L \text{ for each } L\geq n(\lambda).$$
	
\end{definition}

\begin{definition}(Compare \cite[Definition 6]{Gaba1})
	A sequence $(x_n)_{n\geq 1}$ in a $G$-metric space $(X,G)$ is a $\lambda$-sequence if there exist $0<\lambda<1$ and $n(\lambda) \in \mathbb{N}$ such that $$\sum_{i=1}^{L-1} G(x_i,x_{i+1},x_{i+1}) \leq \lambda L \text{ for each } L\geq n(\lambda)+1.$$
	
\end{definition}

\newpage

\section{The results}

\begin{theorem}\label{thmk}
	Let $(X, G)$ be a complete $G$-metric space. Suppose the map $T : X \to X$ satisfies
	
	\begin{equation}\label{condthmk}
	G^{2k+1}(T x, T y, T z) \leq  \lambda G(x, T x, T x)G^k(y, T y, T y)G^k(z, T z, T z)
	\end{equation}
	for all $x, y, z \in X$, where $0 \leq \lambda < 1$, and some $k \in \mathbb{N}$. Then $T$ has a unique fixed point.
\end{theorem}

\begin{proof}
	Let $x_0\in X$ be arbitrary and construct the sequence $\{x_n\}$ such that $x_{n+1}=Tx_n.$ Moreover, we may assume, without loss of generality that $x_n\neq x_{m}$ for $n\neq m$.
	
	For the triplet $(x_n, x_{n+1},x_{n+1})$, and by setting $d_n = G(x_n, x_{n+1},x_{n+1})$, we have:
	
	\begin{align*}
	G^{2k+1}(Tx_{n-1}, Tx_{n},Tx_{n}) & = G^{2k+1}(x_n,x_{n+1},x_{n+1})\\
	&\leq \lambda G(x_{n-1},x_n,x_n) G^k(x_{n},x_{n+1},x_{n+1}) G^k(x_{n},x_{n+1},x_{n+1}).
	\end{align*}
Thus, we have	
	\[ G(x_n, x_{n+1},x_{n+1}) \leq \lambda G(x_{n-1},x_n,x_n) \leq \ldots \leq  \lambda^n G(x_0 , x_1 , x_1 ).  \]
	
Hence for any $m>n,$ we have

\begin{align*}
G(x_n,x_m,x_m) & \leq G(x_n,x_{n+1},x_{n+1})+ G(x_{n+1},x_{n+2},x_{n+2})\\
& +\cdots+  G(x_{m-1},x_{m},x_m)\\
& \leq (q^n + \cdots + q^{m-1})G(x_0 , x_1 , x_1 ) \\
& \leq  \frac{\lambda^{n}}{1-\lambda}G(x_0 , x_1 , x_1 ).
\end{align*}
	
and so $G(x_n,x_m,x_m)\to 0$, as $n, m \to \infty$. Thus $\{x_n \}$ is $G$-Cauchy sequence. Due to the
completeness of $(X, G)$, there exists $u \in X$ such that $\{x_n \}$ $G$-converges to $u$.	
	
On the other hand, using \eqref{condthmk}, we have
\[ G^{2k+1}(Tx_{n-1}, Tx_{n-1},Tu)
\leq qG(x_{n-1}, x_n , x_n ) G^k(x_{n-1}, x_n , x_n )G^k(u, T u, T u)\]

Letting $n \to \infty$, and using the fact that $G$ is continuous on its variable, we get that $G^{2k+1} (u, u, T u) = 0$. Therefore, $T u = u$, hence $u$ is a fixed point of $T$. If $v$ a fixed point of
$T$, then we have

\begin{align*}
G^{2k+1}(u, u, v) &= G^{2k+1}(T u, T u, T v) \\
& \leq  qG(u, T u, T u)G^k(u, T u, T u)G^k(v, T v, T v)\\
& = 0.
\end{align*}
Thus, $u = v$. Now we know the fixed point of $T$ is unique.

\end{proof}

\newpage

\begin{corollary}Let $(X, G)$ be a complete $G$-metric space. Suppose the map $T : X \to X$ satisfies
	
	\begin{equation}\label{condthmkcor}
	G^{2k+1}(T^p x, T^py, T^pz) \leq  qG(x, T^p x, T^p x)G^k(y, T^p y, T^p y)G^k(z, T^p z, T^p z)
	\end{equation}
	for all $x, y, z \in X$, where $0 \leq \lambda < 1$, and some $p,k \in \mathbb{N}$. Then $T$ has a unique fixed point.
\end{corollary}

\begin{proof}
From Theorem \ref{thmk} we know that  $T^p$ has a unique fixed point (say $\xi$), that is,
$T^p \xi = \xi$. Since $T\xi = T T^p \xi = T^{p+1} \xi = T^p T \xi$, so $T \xi$ is
another fixed point for $T^p$ , and by uniqueness, we have $T \xi = \xi$.
\end{proof}

 \begin{theorem}\label{thm1}
 Let $(X, G)$ be a complete $G$-metric space. Suppose the maps $T,P,Q : X \to X$ satisfy
 
 \begin{equation}\label{condthm1}
 G^3(T x, P y, Q z) \leq  \lambda G(x, T x, T x)G(y, P y, P y)G(z, Q z, Q z)
 \end{equation}
 for all $x, y, z \in X$, where $0 \leq \lambda < 1$. Then $T,P$ and $Q$ have a unique common fixed point.
 \end{theorem}

 \begin{proof}
 Let $x_0 \in X$, and define the sequence $\{x_n\}$ by 
 \[
 x_1 = Tx_0, x_2 = Px_1, x_3 = Qx_2 = \cdots,\]
 
 \[ x_{3n+1}= Tx_{3n}, \ \ x_{3n+2} = Px_{3n+1}, \ \ x_{3n+3}= Qx_{3n+2}.
 \]
 
 Using \eqref{condthm1} and assuming, without loss of generality, that $x_n\neq x_m$ for each $n\neq m$, we have

 \begin{align*}
 	G^3(T x_0, P x_1, Q x_2) &=G^3(x_1,x_2,x_3) \\ 
 	                        & \leq 
 \end{align*}
 \[  \leq  \lambda G(x, T x, T x)G(y, P y, P y)G(z, Q z, Q z)  \]
 
 We then conclude that

 \[
 \boxed{G(x_{3n+1},x_{3n+2},x_{3n+3}) \leq  \lambda\  G(x_{3n},x_{3n+1}, x_{3n+2})}.
 \]
 
 In the same manner, it can be shown that 
 
 \[
 \boxed{G(x_{3n+2},x_{3n+3},x_{3n+4}) \leq  \lambda \ G(x_{3n+1},x_{3n+2}, x_{3n+3})} ,
 \]
 
 and 
 \[
 \boxed{G(x_{3n},x_{3n+1},x_{3n+2}) \leq  \lambda \ G(x_{3n-1},x_{3n},x_{3n+1})},
 \]
 so that 
 \begin{align*}
 G(x_{3n},x_{3n+1},x_{3n+2}) \leq & \ \lambda^{3n} G(x_0,x_1,x_2), \\
 G(x_{3n+1},x_{3n+2},x_{3n+3}) \leq & \  \lambda^{3n}G(x_1,x_2,x_3), \\
 G(x_{3n+2},x_{3n+3},x_{3n+4}) \leq & \ \lambda^{3n}G(x_2,x_3,x_4).
 \end{align*}

 Therefore, for all $n$
 \[
 \boxed{G(x_{n},x_{n+1},x_{n+2}) \ \leq \ \lambda^{n} \ r(x_0)},
 \]
 
 with $$r(x_0)=\max\{G(x_0,x_1,x_2),G(x_1,x_2,x_3),G(x_2,x_3,x_4) \}.$$

 Hence for any $l>m>n,$ we have

 \begin{align*}
 G(x_n,x_m,x_m) & \leq G(x_n,x_{n+1},x_{n+1})+ G(x_{n+1},x_{n+2},x_{n+2})\\
 & +\cdots+  G(x_{l-1},x_{l-1},x_l)\\
 & \leq G(x_n,x_{n+1},x_{n+2})+ G(x_{n+1},x_{n+2},x_{n+3})\\
 & +  \cdots+  G(x_{l-2},x_{l-1},x_l)\\
 & \leq  \frac{\lambda^{n}}{1-\lambda}r(x_0).
 \end{align*}

 Similarly, for the cases  $l=m>n,$  and  $l>m=n,$ we have 
 
 \[
 G(x_n,x_m,x_l)  \leq \frac{\lambda^{n-1}}{1-\lambda}r(x_0).
 \]
 
 Thus is $\{x_n\}$ is $G$-Cauchy and hence $G$-converges. Call the limit $\xi.$ From \eqref{condthm1}, we have 
 
 \begin{equation}\label{lim1}
 G(T\xi,x_{3n+2},x_{3n+3})\leq \lambda G(\xi,T\xi,T\xi) G(x_{3n+1},x_{3n+2},x_{3n+2})G(x_{3n+2},x_{3n+3},x_{3n+3})
 \end{equation}
 
 \vspace*{0.5cm}
 
 Taking the limit in \eqref{lim1} as $n\to \infty$, and
 using the fact that the function $G$ is continuous, we obtain
 \[
 G(T\xi, \xi,\xi) \leq 0,
 \]

 that is $ G(T\xi, \xi,\xi)=0$ and hence $T\xi=\xi.$
Similarly, one shows that 
\[
P\xi = \xi = Q\xi.
\]

Moreover, if $\eta$ is a point such that 

\[
\eta= T\eta =P\eta =Q\eta,
\] 
  then from \eqref{condthm1}, we can write
  \[
  G^3(T\xi,P\eta,x_{3n+3}) \leq \lambda G(\xi,T\xi,T\xi)G(\eta,P\eta,P\eta)G(x_{3n+2},x_{3n+3}).
  \]
 Taking the limit, we obtain, 
 \[ \xi=T\xi=P\eta =\eta.   \] 
   Similarly, one gets $\xi=T\xi=Q\eta =\eta$. 
   This completes the proof.
  
  \end{proof}

In the same manner, one easily establish the following:

\begin{theorem}\label{thm2}
	Let $(X, G)$ be a complete $G$-metric space. Let $T_i : X \to X, \ i=1,2,\cdots$ be a family of maps that satisfy
	
	\begin{equation}\label{condthm2}
	G^3(T_ix, T_j y, T_k z) \leq  qG(x, T_i x, T_i x)G(y, T_j y, T_j y)G(z, T_k z, T_k z)
	\end{equation}
	for all $x, y, z \in X$, where $0 \leq q < 1$. Then the $T_i$'s, $i=1,2,\cdots,$ have a unique common fixed point.
\end{theorem}

We conclude this section with the following result, where we make use of the idea of $\alpha$-series.

\begin{theorem}\label{series}
	Let $(X, G)$ be a complete $G$-metric space. Let $T_i : X \to X, \ i=1,2,\cdots$ be a family of maps that satisfy
	
	\begin{equation}\label{condseries}
	G^3(T_ix, T_j y, T_k z) \leq \tensor*[_{k}]{\Delta}{_{i,j}} G(x, T_i x, T_i x)G(y, T_j y, T_j y)G(z, T_k z, T_k z)
	\end{equation}
	for all $x, y, z \in X$, where $0 \leq \tensor*[_{k}]{\Delta}{_{i,j}} < 1$. If the series $\sum_{i=1}^{\infty}\tensor*[_{i+2}]{\Delta}{_{i,i+1}}$ is an $\alpha$-series, then the $T_i$'s, $i=1,2,\cdots,$ have a unique common fixed point.
\end{theorem} 

\begin{proof}
	
	We will proceed in two main steps.
	
	\vspace{0.3cm}
	
	\underline{Claim 1:}  Any fixed point of $T_i$ is also a fixed point of $T_j$ and $T_k$ for $i\neq j\neq k\neq i$.
	
	Indeed, assume that $x^*$ is a fixed point of $T_i$ and suppose that $T_jx^*\neq x^*$ and $T_kx^*\neq x^*$. Then 
	
	\begin{align*}
	G^3(x^*,T_jx^*,T_kx^*) &= G^3(T_ix^*,T_jx^*,T_kx^*) \\
	& \leq  (\tensor*[_{k}]{\Delta}{_{i,j}})G(x^*,x^*,x^*) 
	 G(x^*,T_jx^*,T_jx^*)+ G(x^*,T_kx^*,T_kx^*) ].\\
	\end{align*}
So $G^3(x^*,T_jx^*,T_kx^*)=0$, i.e. $ T_ix^* = x^* =T_jx^*=T_kx^*$.
\vspace{0.3cm}
	
For any $x_0\in X$, we build the sequence $(x_n)$ by setting $x_n= T_n(x_{n-1}),\ n=1,2,\cdots .$ We assume without loss of generality that $x_n\neq x_{m}$ for any $n,m\in \mathbb{N}$.
Using \eqref{condseries}, we obtain

\begin{align*}
G^3(x_n,x_{n+1},x_{n+2}) & = G^3(T_n(x_{n-1}), T_{n+1}(x_{n}),T_{n+2}(x_{n+1})) \\
                    & \leq \left(\tensor*[_{n+2}]{\Delta}{_{n,n+1}}\right)G(x_{n-1},x_n,x_n)G(x_n,x_{n+1},x_{n+1})G(x_{n+1},x_{n+2},x_{n+2})\\
                    &\leq \left(\tensor*[_{n+2}]{\Delta}{_{n,n+1}}\right)G(x_{n-1},x_n,x_{n+1})G(x_n,x_{n+1},x_{n+2})G(x_{n},x_{n+1},x_{n+2})
\end{align*}

by property (G3).

Hence
\[  G(x_n,x_{n+1},x_{n+2}) \leq \left(\tensor*[_{n+2}]{\Delta}{_{n,n+1}}\right)G(x_{n-1},x_n,x_{n+1})  \]

and we obtain recursively

\begin{align*}G(x_n,x_{n+1},x_{n+2}) &\leq \left(\tensor*[_{n+2}]{\Delta}{_{n,n+1}}\right)G(x_{n-1},x_n,x_{n+1})\\
             & \leq \left(\tensor*[_{n+2}]{\Delta}{_{n,n+1}}\right) \left(\tensor*[_{n+1}]{\Delta}{_{n-1,n}}\right) G(x_{n-2},x_{n-1},x_{n})\\
             &\leq \left( \prod_{i=1}^{n} \tensor*[_{i+2}]{\Delta}{_{i,i+1}}\right)G(x_0,x_1,x_2).
\end{align*}
If we set $r_i  = \tensor*[_{i+2}]{\Delta}{_{i,i+1}}$, we have that
$$G(x_{n},x_{n+1},x_{n+2}) \leq \left[ \prod\limits_{i=1}^n r_i \right] G(x_0,x_1,x_2). $$

Therefore, for all $l>m>n$

\begin{align*}
G(x_n,x_{m},x_{l}) & \leq G(x_{n},x_{n+1},x_{n+1})+  G(x_{n+1},x_{n+2},x_{n+2}) + \\
& + \cdots + G(x_{l-1},x_{l-1},x_{l}) \\ 
& \leq G(x_{n},x_{n+1},x_{n+2})+  G(x_{n+1},x_{n+2},x_{n+3}) + \\
& + \cdots + G(x_{l-2},x_{l-1},x_{l}),
\end{align*}

and

\begin{align*}
G(x_n,x_{m},x_{l})& \leq \left(\left[ \prod\limits_{i=1}^n r_i \right] + 
\left[ \prod\limits_{i=1}^{n+1} r_i \right]+ \cdots + \left[ \prod\limits_{i=1}^{n+l-1} r_i \right]\right) G(x_0,x_1,x_2) \\
& = \sum_{k=0}^{l-1}\left[\prod\limits_{i=1}^{n+k}r_i\right]G(x_0,x_1,x_2)\\
& =\sum_{k=n}^{n+l-1}\left[\prod\limits_{i=1}^{k}r_i\right]G(x_0,x_1,x_2).
\end{align*}

Now, let $\lambda$ and $n(\lambda)$ as in Definition \ref{vetro}, then for $n\geq n(\lambda)$ and using the fact that the geometric mean of non-negative real numbers is at most their arithmetic mean, it follows that

\begin{align*}
G(x_n,x_{m},x_{l})& \leq \sum_{k=n}^{n+l-1}\left[\frac{1}{k}\left(\sum_{i=1}^{k}r_i\right)\right]^kG(x_0,x_1,x_2)\\
& =\left(\sum_{k=n}^{n+l-1} \alpha^k\right) G(x_0,x_1,x_2)\\
& \leq \frac{\alpha^n}{1-\alpha}G(x_0,x_1,x_2).
\end{align*}

As $n\to \infty$, we deduce that $G(x_n,x_{m},x_{l}) \to 0.$ Thus $(x_n)$ is a $G$-Cauchy sequence.

Moreover, since $X$ is $G$-complete there exists $u \in X$ such that $(x_n)$ $G$-converges to $u$.

If there exists $n_0$ such that $T_{n_0}u=u$, then by the claim 1, the proof of existence is complete.

Otherwise for any positive integers $k, l$, we have
\begin{align*}
G^3(x_n, T_ku,T_lu)  & = G^3(T_nx_{n-1},T_ku,T_lu) \\
& \leq (\tensor*[_{l}]{\Delta}{_{n,k}})G(x_{n-1},x_n,x_n) G(u,T_ku,T_ku)G(u,T_lu,T_lu).\\
\end{align*}

Letting $n\to \infty$ we obtain
$G^3(u,T_ku,T_lu)=0$, i.e. $ u =T_ku=T_lu$ and $u$ is a common fixed point to the $T_i$'s.
The uniqueness of the fixed point is readily seen from condition \eqref{condseries}.

\end{proof}

The next two results are variants of Theorem \ref{series}. We shall state them without proof as the proofs follow the same technique as the one we just presented.

\begin{theorem}\label{series1}
	Let $(X, G)$ be a complete $G$-metric space. Let $T_i : X \to X, \ i=1,2,\cdots$ be a family of maps that satisfy
	
	\begin{equation}\label{condseries1}
	G^3(T_ix, T_j y, T_k z) \leq \tensor*[_{k}]{\Delta}{_{i,j}} G(x, T_i x, T_i x)G(y, T_j y, T_j y)G(z, T_k z, T_k z)
	\end{equation}
	for all $x, y, z \in X$, where $0 \leq \tensor*[_{k}]{\Delta}{_{i,j}} < 1$. 
If

	\begin{itemize}
		\item[i)] for each $j,k$, $\limsup_{i\to \infty} \tensor*[_{k}]{\Delta}{_{i,j}} <1$ ,
		
		\item[ii)] 
		$$
		\sum_{n=1}^{\infty} C_n <\infty \text{ where } C_n = \prod_{i=1}^{n}r_i = \prod_{i=1}^{n}\tensor*[_{i+2}]{\Delta}{_{i,i+1}},
		$$
		
	\end{itemize}
	then the $T_i$'s have a unique common fixed point in $X$.
	
\end{theorem}

\begin{theorem}\label{series2}
	Let $(X, G)$ be a complete $G$-metric space. Let $T_i : X \to X, \ i=1,2,\cdots$ be a family of maps. Assume that there exist a sequence $(a_n)$ of elements of $X$ such

	\begin{equation}\label{condseries2}
	G^3(T_ix, T_j y, T_k z) \leq \tensor*[_{k}]{\Delta}{_{i,j}} G(x, T_i x, T_i x)G(y, T_j y, T_j y)G(z, T_k z, T_k z)
	\end{equation}
	for all $x, y, z \in X$, where $\tensor*[_{k}]{\Delta}{_{i,j}}:=G(a_i,a_j,a_k)$, and  with $0\leq \tensor*[_{k}]{\Delta}{_{i,j}} <1 , \ i,j,k = 1,2,\cdots$. If the sequence $(r_i)$ where  $r_i = \tensor*[_{i+2}]{\Delta}{_{i,i+1}}$
	is a non-increasing $\lambda$-sequence of $\mathbb{R}^+$ endowed with the $\max\footnote{The max metric $m$ refers to $m(x,y)=\max\{x,y\}$ }$ metric, then the $T_i$'s, $i=1,2,\cdots,$ have a unique common fixed point.
\end{theorem}

The main result of this section (generalization of Theorem \ref{thm2}) then goes as follows:

\begin{theorem}\label{final}
	Let $(X, G)$ be a complete $G$-metric space. Let $T_i : X \to X, \ i=1,2,\cdots$ be a family of maps that satisfy
	
	\begin{equation}\label{condfinal}
	G^{2k+1}(T_ix, T_j y, T_k z) \leq  qG(x, T_i x, T_i x)G^k(y, T_j y, T_j y)G^k(z, T_k z, T_k z)
	\end{equation}
	for all $x, y, z \in X$, where $0 \leq q < 1$ and some $k\in \mathbb{N}$. Then the $T_i$'s, $i=1,2,\cdots,$ have a unique common fixed point.
\end{theorem} 

\begin{proof}
	For any $x_0\in X$, we build the sequence $(x_n)$ by setting $x_n= T_n(x_{n-1}),\ n=1,2,\cdots .$ We assume without loss of generality that $x_n\neq x_{m}$ for any $n,m\in \mathbb{N}$.
	Using \eqref{condfinal}, we obtain

	\begin{align*}
	G^{2k+1}(x_n,x_{n+1},x_{n+2}) & = G^{2k+1}(T_n(x_{n-1}), T_{n+1}(x_{n}),T_{n+2}(x_{n+1})) \\
	& \leq qG(x_{n-1},x_n,x_n)G^k(x_n,x_{n+1},x_{n+1})G^k(x_{n+1},x_{n+2},x_{n+2})\\
	&\leq q G(x_{n-1},x_n,x_{n+1})G^k(x_n,x_{n+1},x_{n+2})G^k(x_{n},x_{n+1},x_{n+2})
	\end{align*}
by property (G3).

Hence

\begin{equation}\label{key1} 
G(x_n,x_{n+1},x_{n+2}) \leq qG(x_{n-1},x_n,x_{n+1}), 	
	\end{equation}
By usual procedure from \eqref{key1}, since $q<1$, it follows that $\{x_n\}$ is a $G$-Cauchy sequence. By $G$-completeness of $X$, there exists $x^*\in X$ such that $\{x_n\}$ $G$-converges to $x^*.$	
	
For any positive integers $k, l$, we have
\begin{align*}
G^{2k+1}(x_n, T_ku,T_lu)  & = G^{2k+1}(T_nx_{n-1},T_ku,T_lu) \\
& \leq q G(x_{n-1},x_n,x_n) G^k(u,T_ku,T_ku)G^k(u,T_lu,T_lu).\\
\end{align*}

Letting $n\to \infty$ we obtain
$G^{2k+1}(u,T_ku,T_lu)=0$, i.e. $ u =T_ku=T_lu$ and $u$ is a common fixed point to the $T_i$'s.
The uniqueness of the fixed point is readily seen from condition \eqref{condfinal}.	
	
\end{proof}

\begin{remark}
	Variants of Theorem \ref{final} can also be easily formulated using ideas from Theorem \ref{series1} and Theorem \ref{series2} respectively. This will be left to the reader.
\end{remark}

\begin{theorem}\label{end}
	Let $(X, G)$ be a complete $G$-metric space. Let $T : X \to X$ be a $G$-continuous mapping that satisfies
	
	\begin{equation}\label{condend}
	G^{2k+1}(T^{2k-1}x, T^{2k} x, T^{2k+1} x) \leq  q A_k(x)
	\end{equation}
where $$A_k(x)= 
G(T^{2k-2}x, T^{2k-1} x, T^{2k-1}x)G^k(T^{2k-1}x, T^{2k} x, T^{2k}x) G^k(T^{2k}x, T^{2k+1} x, T^{2k+1}x)$$	for $x \in X$, where $0 \leq q < 1$ and some $k\in \mathbb{N}$. Then $T$ has a fixed point.
\end{theorem}

\begin{proof}
	For any $x_0\in X$, we build the sequence $(x_n)$ by setting $x_n= T_n(x_{n-1}),\ n=1,2,\cdots .$ We assume without loss of generality that $x_n\neq x_{m}$ for any $n,m\in \mathbb{N}$.
	
First, observe that for $n\in \mathbb{N},$ large enough, we can find $l\in \mathbb{N}$ such that $x_n= T^{2k+1}x_l$. So 

\begin{align*}
G^{2k+1}(x_{n}, x_{n+1},  x_{n+2})& =G^{2k+1}(T^{2k-1}x_l, T^{2k} x_l, T^{2k+1} x_l) \\
& \leq q A_k(x_l).
\end{align*}	
	
	Moreover,
	\begin{align*}
A_k(x_l) & =G(T^{2k-2}x_l, T^{2k-1} x_l, T^{2k-1}x_l)G^k(T^{2k-1}x_l, T^{2k} x_l, T^{2k}x_l) G^k(T^{2k}x_l, T^{2k+1} x_l, T^{2k+1}x_l)\\
      & = G(x_{n-1},x_n,x_n) G^k(x_n,x_{n+1},x_{n+1})G^k(x_{n+1},x_{n+2},x_{n+2}) \\
      & \leq G(x_{n-1},x_n,x_{n+1}) G^k(x_n,x_{n+1},x_{n+2})G^k(x_{n},x_{n+1},x_{n+2})
	\end{align*}

	by property (G3).
	
	Hence
	
	\begin{equation}\label{key12} 
	G(x_n,x_{n+1},x_{n+2}) \leq qG(x_{n-1},x_n,x_{n+1}), 	
	\end{equation}
	By usual procedure from \eqref{key12}, since $q<1$, it follows that $\{x_n\}$ is a $G$-Cauchy sequence. By $G$-completeness of $X$, there exists $x^*\in X$ such that $\{x_n\}$ $G$-converges to $x^*.$
	Furthermore, since $T$ is
	$G$-continuous, from $x_{n+1} = T x_n$ , letting $n \to \infty$ at both sides, we have $x^* = T x^*$. Thus,
	$x^*$ is a fixed point of $T$.

\end{proof}

\begin{theorem}\label{end1}
	Let $(X, G)$ be a complete $G$-metric space. Let $T : X \to X$ be a mapping that satisfies
	
	\begin{equation}\label{condend1}
	G^{2k+1}(T^{2k-1}x, T^{2k} y, T^{2k+1} z) \leq  q A_k(x)
	\end{equation}
	where $$A_k(x)= 
	G(T^{2k-2}x, T^{2k-1} x, T^{2k-1}x)G^k(T^{2k-1}y, T^{2k} y, T^{2k}y) G^k(T^{2k}z, T^{2k+1} z, T^{2k+1}z)$$	for $x,y,z \in X$, where $0 \leq q < 1$ and some $k\in \mathbb{N}$. Then $T$ has a unique fixed point.
\end{theorem}

\bibliographystyle{amsplain}

\end{document}